\documentclass{amsart}

\usepackage{amsmath}
\usepackage{amssymb}

\newtheorem{theorem}{Theorem}[section]
\newtheorem{lemma}[theorem]{Lemma}

\newtheorem{corollary}{Corollary}[section]

\theoremstyle{definition}

\theoremstyle{remark}

\newtheorem*{claim}{Claim}

\numberwithin{equation}{section}

\newcommand{\abs}[1]{\lvert#1\rvert}

\begin{document}

\title[Combinatorial Principles for trees]{SOME COMBINATORIAL PRINCIPLES FOR TREES AND APPLICATIONS TO TREE-FAMILIES IN
BANACH SPACES}

\author{Costas Poulios}

\address{Department of Mathematics, University of Athens, 15784, Athens, Greece}

\email{costas314@gmail.com}

\author{Athanasios Tsarpalias}
\address{Department of Mathematics, University of Athens, 15784, Athens, Greece}
\email{atsarp@math.uoa.gr}


\date{}


\begin{abstract}
Suppose that $(x_s)_{s\in S}$ is a normalized family in a Banach
space indexed by the dyadic tree $S$. Using Stern's combinatorial
theorem we extend important results from sequences in Banach spaces
to tree-families. More precisely, assuming that for any infinite
chain $\beta$ of $S$ the sequence $(x_s)_{s\in\beta}$ is weakly
null, we prove that there exists a subtree $T$ of $S$ such that for
any infinite chain $\beta$ of $T$ the sequence $(x_s)_{s\in\beta}$
is nearly (resp., convexly) unconditional. In the case where
$(f_s)_{s\in S}$ is a family of continuous functions, under some
additional assumptions, we prove the existence of a subtree $T$ of
$S$ such that for any infinite chain $\beta$ of $T$, the sequence
$(f_s)_{s\in\beta}$ is unconditional. Finally, in the more general
setting where for any chain $\beta$, $(x_s)_{s\in\beta}$ is a
Schauder basic sequence, we obtain a dichotomy result concerning the
semi-boundedly completeness of the sequences $(x_s)_{s\in\beta}$.
\end{abstract}

\maketitle

\section{Introduction} In a well-known example B. Maurey and H. P. Rosenthal
\cite{Maurey} showed that if $(x_n)_{n\in\mathbb{N}}$ is a
normalized weakly null sequence in a Banach space then we could not
expect that $(x_n)$ admits an unconditional subsequence. Further, W.
T. Gowers and B. Maurey \cite{Gowers} exhibited a Banach space not
containing any unconditional basic sequence.

Despite the aforementioned constructions there are some positive
results where either some special sequences $(x_n)$ are considered
or weaker forms of unconditionality appear. More precisely, H. P.
Rosenthal proved the following theorem.

\begin{theorem}\label{th.Rosenthal}
Let $K$ be a compact Hausdorff space and let
$(f_n)_{n\in\mathbb{N}}$, $f_n:K\to\mathbb{R}$, be a sequence of
non-zero, continuous, characteristic functions. If
$(f_n)_{n\in\mathbb{N}}$ converges pointwise to zero, then it
contains an unconditional basic subsequence.
\end{theorem}

Although the initial proof uses transfinite induction, the nature of
the previous result is purely combinatorial. Indeed, the proof of
Theorem \ref{th.Rosenthal} (see \cite{Godefroy} and \cite{Odell})
can be obtained from the next result which in turn depends on the
infinite Ramsey theorem. In the following, if $M$ is an infinite
subset of $\mathbb{N}$, $[M]^\omega$ denotes the set of all infinite
subsets of $M$.

\begin{theorem}\label{th.Rosenthal.comb.}
Let $\mathcal{F}\subset \mathcal{P}(\mathbb{N})$ be a compact family
of finite subsets of $\mathbb{N}$. Then for any
$N\in[\mathbb{N}]^\omega$, there exists $M\in[N]^\omega$ such that
the family $\mathcal{F}[M]=\{F\cap M \mid F\in\mathcal{F}\}$ is
hereditary (that is, if $A\subset B$ and $B\in \mathcal{F}[M]$, then
$A\in \mathcal{F}[M]$).
\end{theorem}

As a matter of fact, it was infinite Ramsey theory which led to a
series of positive results. One of them was obtained by J. Elton
\cite{Elton} (see also \cite{Odell}).

\begin{theorem}\label{th.Elton}
Every normalized weakly null sequence in a Banach space contains a
nearly unconditional subsequence.
\end{theorem}

The notion of nearly unconditionality concerns the unconditional
behavior of linear combinations with coefficients bounded away from
zero. The precise definition is the following: A normalized sequence
$(x_n)$ in a Banach space is called \emph{nearly unconditional} if
for every $\delta>0$ there exists $C=C(\delta)>0$ such that for any
$n\in\mathbb{N}$, any scalars $a_1,\ldots,a_n\in[-1,1]$ and any
$F\subseteq\{i\le n \mid |a_i|>\delta\}$,
$$\Big\|\sum_{i\in F}a_ix_i\Big\|\le
C(\delta)\Big\|\sum^n_{i=1}a_ix_i\Big\|.$$

Using Ramsey's theory in a very elegant way, Elton proved the
following principle from which Theorem \ref{th.Elton} is obtained.

\begin{theorem}\label{th.Elton.comb.}
Let $F$ be a weakly compact subset of the unit ball of $c_0$ and let
$\delta>0$ and $\epsilon\in(0,1)$ be given. Then for every
$N\in[\mathbb{N}]^\omega$, there exists
$M=\{m_i\}_{i=1}^\infty\in[N]^\omega$ such that:

for every $f\in F$, $n\in\mathbb{N}$ and $I\subseteq \{i\leq n \mid
f(m_i)>0\}$ with $\sum_{i\in I}f(m_i)>\delta$ there exists $g\in F$
such that
\begin{enumerate}
\renewcommand{\theenumi}{\roman{enumi}}
    \item $\sum_{i\in I}g^+(m_i)>(1-\epsilon)\sum_{i\in I}f(m_i)$
    [where $g^+=\max(g,0)$]
    \item $\sum_{i\in J}|g(m_i)|<\epsilon\sum_{i\in I}f(m_i)$, where
    $J=\{i\leq n \mid i\notin I \text{ or } g(m_i)<0\}$.
\end{enumerate}
\end{theorem}

The subsequences of a weakly null sequence have also been
investigated with respect to the property of convex
unconditionality. A normalized sequence $(x_n)$ in a Banach space is
called \emph{convexly unconditional} if for every $\delta>0$ there
exists $C=C(\delta)>0$ such that for any absolutely convex
combination $x=\sum^\infty_{n=1}a_nx_n$ with $\|x\|\ge\delta$ and
any sequence $(\varepsilon_n)_{n\in\mathbb{N}}$ of signs,
$$\Big\|\sum^\infty_{n=1}\varepsilon_na_nx_n\Big\|\ge C(\delta).$$
The next result, concerning the case of convex unconditionality, has
been proved by S. A. Argyros, S. Mercourakis and A. Tsarpalias
\cite{Arg}.

\begin{theorem}\label{th.AMT}
Every normalized weakly null sequence in a Banach space contains a
convexly unconditional subsequence.
\end{theorem}

As in the previous cases, the proof of this theorem is based on the
next combinatorial principle.

\begin{theorem}\label{th.AMT.comb.}
Let $F$ be a weakly compact subset of $c_0$ and let $\delta>0$ and
$\epsilon\in(0,1)$ be given. Then for every
$N\in[\mathbb{N}]^\omega$, there exists
$M=\{m_i\}_{i=1}^\infty\in[N]^\omega$ such that:

for every $f\in F$, $n\in\mathbb{N}$ and $I\subseteq \{
1,2,\ldots,n\}$ with $\min_{i\in I}f(m_i)>\delta$ there exists $g\in
F$ satisfying the conditions
\begin{enumerate}
\renewcommand{\theenumi}{\roman{enumi}}
    \item $\min_{i\in I}g(m_i)>(1-\epsilon)\delta$
    \item $\sum_{i\leq n, i\notin I}|g(m_i)|<\epsilon\delta$.
\end{enumerate}
\end{theorem}

Finally, the combinatorial proof of Rosenthal's theorem has been
expanded and some stronger results have been obtained. As pointed
out in \cite{Arv}, Theorem \ref{th.Rosenthal} can not be extended in
the case where the range of $f_n$ is a finite set of arbitrarily
large cardinality. However, A. D. Arvanitakis \cite{Arv} expanded
this theorem in the case where the cardinality of the range of $f_n$
is finite and uniformly bounded by some positive integer.

\begin{theorem}\label{th.Arvan.1}
Let $K$ be a Hausdorff compact space, $X$ a Banach space and
$(f_n)_{n\in\mathbb{N}}$, $f_n:K\to X$, a normalized sequence of
continuous functions. We assume that $(f_n)_{n\in\mathbb{N}}$
converges pointwise to zero and that the range of $f_n$'s is of
finite cardinality uniformly bounded by some positive integer $J$.
Then $(f_n)_{n\in\mathbb{N}}$ contains an unconditional subsequence.
\end{theorem}

The following result, also proved in \cite{Arv}, concerns the case
where the space $X$ in the above theorem is finite dimensional.

\begin{theorem}\label{th.Arvan.2}
Let $K$ be a Hausdorff compact space and $(f_n)_{n\in\mathbb{N}}$,
$f_n:K\to \mathbb{R}^m$, a uniformly bounded sequence of continuous
functions which converges pointwise to zero. We also assume that
there are a null sequence $(\epsilon_n)_{n\in\mathbb{N}}$ of
positive numbers and a positive real number $\mu$ such that for
every $n\in\mathbb{N}$ and any $x\in K$ either $\|f_n(x)\|\leq
\epsilon_n$ or $\|f_n(x)\|\geq \mu$. Then $(f_n)_{n\in\mathbb{N}}$
contains an unconditional subsequence.
\end{theorem}

The above theorems are derived by the following combinatorial
principle (see also \cite{Arv}) extending Theorem
\ref{th.Rosenthal.comb.}.

\begin{theorem}\label{th.Arvan.comb.}
Assume that $I$ is a set, $n$ a positive integer and for any $i\in
I$, $F_1^i,\ldots ,F_n^i$ are finite subsets of $\mathbb{N}$ such
that setting $F^i=\cup_{k=1}^nF_k^i$, the closure of the family
$\mathcal{F}=\{F^i\mid i\in I\}$ in the pointwise topology contains
only finite sets. Then for any $N\in[\mathbb{N}]^\omega$, there
exists $M=\{m_1<m_2<\ldots\}\in[N]^\omega$ such that the following
holds:

given $i\in I$, $q\in \mathbb{N}$, $k\in\{1,\ldots,n\}$ and
$A\subseteq F_k^i\cap\{m_1,\ldots,m_q\}$ then there exists $i'\in I$
such that $F_k^{i'}\cap\{m_1,\ldots,m_q\}=A$ and
$F_{k'}^{i'}\cap\{m_1,\ldots,m_q\}\subseteq A$ for any $k'\neq k$.
\end{theorem}

Throughout this paper $S$ denotes the standard dyadic tree, that is
the set $S=\cup_{n=0}^\infty\{0,1\}^n$ of all finite sequences in
$\left\{0,1\right\}$, including the empty sequence denoted by
$\emptyset$. The elements $s\in S$ are called \emph{nodes}. If $s$
is a node and $s\in \left\{0,1\right\}^{n}$, we say that $s$ is on
the \emph{$n$-th level of $S$}. We denote the level of a node $s$ by
$\text{lev}(s)$. The initial segment partial ordering on $S$ is
denoted by $\leq$ and we write $s<s'$ if $s\leq s'$ and $s\neq s'$.
If $s\leq s'$, we say $s'$ is a \emph{follower} of $s$ while if
$s,s'$ are nodes such that neither $s\leq s'$ nor $s'\leq s$ then
$s$ and $s'$ are called incomparable. We also say that the nodes
$s\cup\{0\}$ and $s\cup\{1\}$ are the \emph{successors} of the node
$s$.

A partially ordered set $T$ is called a \emph{dyadic tree} if it is
order isomorphic to $(S,\leq)$. A \emph{subtree} of $S$ is any
subset $S'$ of $S$ which has a single minimal element and any
element of $S'$ has exactly two successors. In the sequel we mean by
a tree always a dyadic tree. If $T$ is a tree, a \emph{chain} of $T$
is an infinite linearly ordered subset of $T$. Throughout this
paper, $\mathcal{C}(T)$ denotes the set of all chains of $T$.

The set $\mathcal{C}(T)$ is endowed with the relative topology of
the product topology of $\mathcal{P}(T)$. A sub-basis of this
topology consists of the sets $U(t)=\left\{\beta \in
\mathcal{C}(T)\mid t\in\beta \right\}$ and $V(t)=\left\{\beta \in
\mathcal{C}(T) \mid t\notin\beta \right\}$ where $t$ varies over the
elements of $T$. Clearly the sets $U(t),V(t)$ are open and closed.
It is also known (see \cite{Stern}) that $\mathcal{C}(T)$ is a
$G_{\delta}$ subset of $\mathcal{P}(T)$. Therefore (see
\cite{Kechris}) the topology of $\mathcal{C}(T)$ is induced by a
complete metric.

Considering on $\mathcal{C}(T)$ the topology described above, J.
Stern \cite{Stern} proved the following Ramsey-type theorem for the
dyadic tree. Furthermore, Stern (see also \cite{Stern}) applied his
combinatorial result in the theory of Banach spaces and extended
Rosenthal's $\ell_1$-theorem \cite{Rosenthal} to the case of
tree-families.

\begin{theorem} \label{th.Stern}
Let $T$ be a tree and let $\mathcal{A}\subseteq \mathcal{C}(T)$ be
an analytic set of chains. There exists a subtree $T'$ of $T$ such
that either $\mathcal{C}(T')\subseteq \mathcal{A}$ or
$\mathcal{C}(T') \cap \mathcal{A}= \emptyset$.
\end{theorem}

Stern's initial proof uses forcing methods. However, C. Ward Henson
(see \cite{Odell}) observed that the above mentioned theorem follows
from some significant results of Ramsey theory.

In this paper, using Stern's theorem we prove in Section
\ref{sec.Combinatorial-Trees} some combinatorial principles for the
dyadic tree. These results combine combinatorial methods with
concepts and techniques coming from Analysis and extend Theorems
\ref{th.Rosenthal.comb.}, \ref{th.Elton.comb.}, \ref{th.AMT.comb.}
and \ref{th.Arvan.comb.} mentioned in this introduction.

In Sections \ref{sec.Continuous Functions}, \ref{sec.Nearly} and
\ref{sec.Convex}, we consider normalized families $(x_s)_{s\in S}$
of elements of a Banach space indexed by the dyadic tree $S$, such
that for any chain $\beta$ of $S$ the sequence $(x_s)_{s\in \beta}$
is weakly null. Our aim is to investigate the unconditional
behaviour of the sequences $(x_s)_{s\in \beta}$ for
$\beta\in\mathcal{C}(S)$. More precisely, in Section
\ref{sec.Continuous Functions}, we prove that in some special cases
there exists a subtree $T$ of $S$ such that for any chain $\beta$ of
$T$ the sequence $(x_s)_{s\in \beta}$ is unconditional. These
results extend Rosenthal's and Arvanitakis' theorems. In Section
\ref{sec.Nearly}, we show that there always exists a subtree $T$ of
$S$ such that all the sequences $(x_s)_{s\in \beta}$,
$\beta\in\mathcal{C}(T)$, are nearly unconditional. In Section
\ref{sec.Convex}, we also prove the existence of a subtree $T$ of
$S$ such that all the sequences $(x_s)_{s\in \beta}$,
$\beta\in\mathcal{C}(T)$, are convexly unconditional. These results
extend Theorems \ref{th.Elton} and \ref{th.AMT} respectively.

Finally, in Section \ref{sec.general} we consider the more general
case where $(x_s)_{s\in S}$ is a normalized tree-family such that
for any chain $\beta\in\mathcal{C}(S)$, $(x_s)_{s\in\beta}$ is a
Schauder basic sequence. In this framework we prove the following
dichotomy result (see Theorem \ref{th.general}): there always exists
a subtree $T$ of $S$ such that either (a)~for any chain
$\beta\in\mathcal{C}(T)$, $(x_s)_{s\in\beta}$ is semi-boundedly
complete, or (b)~for no chain $\beta\in\mathcal{C}(T)$,
$(x_s)_{s\in\beta}$ is semi-boundedly complete. Furthermore, if we
assume that $(x_s)_{s\in\beta}$ is weakly null for any chain
$\beta\in\mathcal{C}(S)$, then we can combine the above dichotomy
with the results of Section \ref{sec.Nearly} and we obtain the next
stronger result (see Theorem \ref{th.dichotomy-c0}): there always
exists a subtree $T$ of $S$ such that either (a)~for any chain
$\beta\in\mathcal{C}(T)$, $(x_s)_{s\in\beta}$ is semi-boundedly
complete, or (b)~for any chain $\beta\in\mathcal{C}(T)$,
$(x_s)_{s\in\beta}$ is $C$-equivalent to the unit vector basis of
$c_0$, where $C>0$ is a common constant. It is worth mentioning that
the proof of Theorem \ref{th.general} uses analytic sets. Actually,
this is the only point where we appeal to the full strength of
Stern's theorem. In all the other cases, the sets appearing in the
proofs are Borel sets.

In order to prove the main results of Sections \ref{sec.Continuous
Functions}, \ref{sec.Nearly} and \ref{sec.Convex}, we can rely on
the combinatorial principles of Section
\ref{sec.Combinatorial-Trees} and transfer the arguments of
\cite{Arv}, \cite{Elton} and \cite{Arg} respectively in the more
complicated setting of tree-families. However the proofs obtained
are quite long and technical and they will not be presented. Instead
our approach uses the corresponding results for sequences and
Stern's theorem and provides us with short proofs of the theorems
contained in this paper. Although we do not use the principles of
Section \ref{sec.Combinatorial-Trees}, we think that they are of
independent interest and they point out the underlying combinatorial
nature of the main results of this work.

In conclusion, the purpose of the present work is to extend
important results from sequences to tree-families. The essential
attitude lying in the core of the paper is that in general this
passage from sequences to tree-families is fundamental in Analysis,
Set Theory and Logic. This passage is not trivial and usually
requires new ideas and techniques. We further believe that the ideas
contained in our proofs can be applied in more general concepts.

\section{Some combinatorial principles for the dyadic
tree}\label{sec.Combinatorial-Trees} In the following $bcn(S)$
denotes the set of all functions $f:S\to\mathbb{R}$ such that $f$ is
bounded and for any chain $\beta$ of $S$ the sequence
$(f(s))_{s\in\beta}$ converges to zero. Clearly, $bcn(S)$ is a
linear subspace of the space $\ell_\infty(S)$ of bounded real
functions defined on $S$. Further, we consider on $bcn(S)$ the
topology of pontwise convergence, that is the relative topology of
the product topology of $\mathbb{R}^S$. Since $S$ is countable,
$bcn(S)$ is metrizable.

In this section, we first prove the following theorem which expands
Elton's combinatorial principle.

\begin{theorem}\label{th.Tree-Elton-Comb.}
Suppose that $F$ is a compact subset of $bcn(S)$, with $F\subseteq
B(\ell_\infty(S))$, where $B(\ell_\infty(S))=\{f:S\to\mathbb{R} \mid
\|f\|_\infty\leq 1\}$ and let $\delta>0$ and $\epsilon\in(0,1)$ be
given. Then there exists a subtree $T$ of $S$ which satisfies the
following property:

for any $\beta=(s_i)\in\mathcal{C}(T)$, any $f\in F$, any
$n\in\mathbb{N}$ and any $I\subseteq\{i\le n\mid f(s_i)>0\}$ with
$\sum_{i\in I}f(s_i)>\delta$, there exists $g\in F$ such that
\begin{enumerate}
\renewcommand{\theenumi}{\roman{enumi}}
\item  $\sum_{i\in I}g^+(s_i)>(1-\epsilon)\sum_{i\in I}f(s_i)$
[where $g^+=\max(g,0)$],
\item  $\sum_{i\in J}|g(s_i)|<\epsilon\sum_{i\in I}f(s_i)$,
where $J=\{i\le n \mid i\notin I$ or $g(s_i)<0\}$.
\end{enumerate}
\end{theorem}

Roughly speaking, for any chain $\beta$ of $T$ and any $f\in F$ we
find a function $g\in F$ which preserves the positive $\ell_1$ mass
of $f$ on the finite set $I$ and is very close to zero on the other
$s_i$'s. If we could find a function $g\in F$ such that
$g(s_i)=f(s_i)$ for $i\in I$ and $g(s_i)=0$ for $i\notin I$, $i\leq
n$, then it would follow (see \cite{Odell}) that for any normalized
family $(x_s)_{s\in S}$ so that $(x_s)_{s\in \beta}$ is weakly null
for any $\beta\in\mathcal{C}(S)$ there is a subtree $T$ of $S$ such
that $(x_s)_{s\in \beta}$ is unconditional for any
$\beta\in\mathcal{C}(T)$, which of course is not true.

\begin{proof}[Proof of Theorem \ref{th.Tree-Elton-Comb.}]
Let $F_0$ be a countable dense subset of $F$. Consider the set
$\mathcal{A}$ of all chains
$\beta=(s_i)_{i=1}^\infty\in\mathcal{C}(S)$ which satisfy the
following property: for any $f\in F_0$, any $n\in\mathbb{N}$ and any
$I\subseteq\{1,2,\ldots,n\}$ with $f(s_i)>0$, $i\in I$, and
$\sum_{i\in I}f(s_i)>\delta$, there is $g\in F_0$ such that
\begin{enumerate}
\item  $\sum_{i\in I}g^+(s_i)>(1-\frac{\epsilon}{2})\sum_{i\in I}f(s_i)$,
\item  $\sum_{i\in J}|g(s_i)|<\frac{\epsilon}{2}\sum_{i\in I}f(s_i)$,
where $J=\{i\le n \mid i\notin I$ or $g(s_i)<0\}$.
\end{enumerate}

\begin{claim}
The set $\mathcal{A}$ is a Borel subset of $\mathcal{C}(S)$.
\end{claim}

Indeed, we have
$$\mathcal{A}=\bigcap_{f\in F_0}\bigcap_{n\in\mathbb{N}}
\bigcap_{I\subseteq\{1,\ldots,n\}} \Big[E(f,n,I) \bigcup
\Big(\bigcup_{g\in F_0}D(f,n,I,g)\Big)\Big]$$ where
$$E(f,n,I)=\Big\{\beta=(s_i)\in\mathcal{C}(S) \mid f(s_i)\leq 0 \text{ for
some } i\in I \text{ or } \sum_{i\in I}f(s_i)\le \delta\Big\}$$
\begin{flushright}
$D(f,n,I,g)=\Big\{\beta=(s_i)\in\mathcal{C}(S) \mid f(s_i)>0$ for
all $i\in I, \,\,  \sum_{i\in I}f(s_i)> \delta$ and $g$ satisfies
properties (1) and (2) in the definition of $\mathcal{A}\Big\}$.
\end{flushright}
Clearly, $E(f,n,I)$ and $D(f,n,I,g)$ are open subsets of
$\mathcal{C}(S)$, therefore $\mathcal{A}$ is a Borel set.

Now Stern's theorem implies that there is a subtree $T$ of $S$ such
that either (a)~$\mathcal{C}(T)\subseteq \mathcal{A}$ or
(b)~$\mathcal{C}(T)\cap \mathcal{A}=\emptyset$. However, the case
(b) can be excluded. Indeed, let us assume that $\mathcal{C}(T)\cap
\mathcal{A}=\emptyset$ and let $\beta$ be any chain of $T$. Applying
Theorem \ref{th.Elton.comb.}, we find a subchain $\alpha=(s_i)$ of
$\beta$ such that for any $f\in F$, any $n\in\mathbb{N}$ and any
$I\subseteq\{i\leq n\mid f(s_i)>0\}$ with $\sum_{i\in
I}f(s_i)>\delta$, there is $g\in F$ such that $\sum_{i\in
I}g^+(s_i)>(1-\frac{\epsilon}{2})\sum_{i\in I}f(s_i)$ and
$\sum_{i\in J}|g(s_i)|<\frac{\epsilon}{2}\sum_{i\in I}f(s_i)$, where
$J=\{i\le n \mid i\notin I$ or $g(s_i)<0\}$. Since $F_0$ is dense in
$F$, it follows that we can find $g\in F_0$ satisfying the above
properties. Hence, $\alpha\in\mathcal{C}(T)\cap \mathcal{A}$ and we
have reached a contradiction. Therefore, $\mathcal{C}(T)\subseteq
\mathcal{A}$.

Since $F_0$ is dense in $F$, we can easily verify that the subtree
$T$ satisfies the conclusion of the theorem.
\end{proof}

In a similar method we also prove the next combinatorial theorem for
the dyadic tree.

\begin{theorem}\label{th.Tree-AMT-Comb.}
Suppose that $F$ is a compact subset of $bcn(S)$ which is bounded
with respect to the supremum norm and let $\delta>0$ and
$\epsilon\in(0,1)$ be given. Then there exists a subtree $T$ of $S$
satisfying the following property:

for any chain $\beta=(s_i)$ of $T$, any $f\in F$, any
$n\in\mathbb{N}$ and any $I\subseteq\{1,\ldots,n\}$ with $\min_{i\in
I}f(s_i)>\delta$, there exists $g\in F$ such that
\begin{enumerate}
\renewcommand{\theenumi}{\roman{enumi}}
\item $\min_{i\in I}g(s_i)>(1-\epsilon)\delta$
\item $\sum_{i\notin I,i\le n}|g(s_i)|<\epsilon\delta$.
\end{enumerate}
\end{theorem}

\begin{proof}
Let $F_0$ be a countable dense subset of $F$ and let $\mathcal{A}$
be the set of all chains $\beta=(s_i)\in\mathcal{C}(S)$ which
satisfy the following: for any $f\in F_0$, any $n\in\mathbb{N}$ and
any $I\subseteq\{1,2,\ldots,n\}$ with $\min_{i\in I}f(s_i)>0$ there
is $g\in F_0$ such that $\min_{i\in I}g(s_i)>(1-\epsilon)\delta$ and
$\sum_{i\notin I,i\le n}|g(s_i)|<\epsilon\delta$. Then,
$$\mathcal{A}=\bigcap_{f\in F_0}\bigcap_{n\in\mathbb{N}}
\bigcap_{I\subseteq\{1,\ldots,n\}} \Big[E(f,n,I) \bigcup
\Big(\bigcup_{g\in F_0}D(f,n,I,g)\Big)\Big]$$ where
$$E(f,n,I)=\Big\{\beta=(s_i)\in\mathcal{C}(S) \mid \min_{i\in I}f(s_i)\leq
\delta \Big\}$$
\begin{flushright}
$D(f,n,I,g)=\Big\{\beta=(s_i)\in\mathcal{C}(S) \mid \min_{i\in
I}f(s_i)>\delta$, $\min_{i\in I}g(s_i)>(1-\epsilon)\delta$ and
$\sum_{i\notin I,i\le n}|g(s_i)|<\epsilon\delta$ $\Big\}$.
\end{flushright}
It follows that $\mathcal{A}$ is a Borel subset of $\mathcal{C}(S)$.
Stern's theorem implies that there is a subtree $T$ of $S$ such that
either (a)~$\mathcal{C}(T) \subseteq \mathcal{A}$ or
(b)~$\mathcal{C}(T) \cap \mathcal{A} =\emptyset$. By Theorem
\ref{th.AMT.comb.}, the case (b) is excluded. Therefore,
$\mathcal{C}(T) \subseteq \mathcal{A}$ and the result follows.
\end{proof}

Finally, we expand Theorem \ref{th.Arvan.comb.} to obtain a
combinatorial theorem for trees.

\begin{theorem}\label{th.Tree-Arvan-Comb.}
Assume that $I$ is a set, $n$ a positive integer and for every $i\in
I$, $F_1^i,\ldots,F_n^i$are finite subsets of $S$. For any $i\in I$
we set $F^i=\cup_{k=1}^nF_k^i$ and let $\mathcal{F}=\{F^i\mid i\in
I\}$. We also assume that for any $F$ in the closure of
$\mathcal{F}$ and for any chain $\beta\in\mathcal{C}(S)$, the set
$F\cap\beta$ is finite. Then there exists a subtree $T$ of $S$
satisfying the following property:

for any chain $\beta=(s_i)\in \mathcal{C}(T)$, any
$F^i\in\mathcal{F}$, any $q\in\mathbb{N}$, any
$k\in\{1,2,\ldots,n\}$ and any $A\subseteq
F_k^i\cap\{s_1,\ldots,s_q\}$ there exists $F^{i'}\in\mathcal{F}$
such that:
\begin{enumerate}
\renewcommand{\theenumi}{\roman{enumi}}
\item $F_k^{i'}\cap\{s_1,\ldots,s_q\}=A$
\item $F_{k'}^{i'}\cap\{s_1,\ldots,s_q\}\subseteq A$ for any $k'\neq
k$.
\end{enumerate}
\end{theorem}

\begin{proof}
The powerset $\mathcal{P}(S)$ endowed with the product topology is a
compact metric space. Therefore, $\mathcal{P}(S)$ is separable. It
follows that for any $k=1,2,\ldots,n$ there is a countable subset
$I_k\subset I$ such that $\{F_k^i\}_{i\in I_k}$ is a dense subset of
$\{F_k^i\}_{i\in I}$. We set $J=\cup_{k=1}^nI_k$. Clearly, $J$ is
countable and for every $k=1,2,\ldots,n$, $\{F_k^i\}_{i\in J}$ is
dense in $\{F_k^i\}_{i\in I}$.

We consider the set $\mathcal{A}$ of all chains
$\beta=(s_i)\in\mathcal{C}(S)$ which satisfy the following property:
for any $i\in J$, any $q\in\mathbb{N}$, any $k\in\{1,2,\ldots,n\}$
and any $A\subseteq F_k^i\cap\{s_1,\ldots,s_q\}$ there is $i'\in J$
such that: $F_k^{i'}\cap\{s_1,\ldots,s_q\}=A$ and
$F_{k'}^{i'}\cap\{s_1,\ldots,s_q\}\subseteq A$ for any $k'\neq k$.
Then we have
$$\mathcal{A}=\bigcap_{i\in J}\bigcap_{q\in\mathbb{N}}
\bigcap_{k\in\{1,\ldots,n\}} \bigcap_{A\subset F_k^i}
\Big[E(i,q,k,A) \bigcup \Big(\bigcup_{i'\in
J}D(i,q,k,A,i')\Big)\Big]$$ where
\begin{flushleft}
$E(i,q,k,A)=\Big\{\beta=(s_i)\in\mathcal{C}(S) \mid A\nsubseteq
F_k^i\cap\{s_1,\ldots,s_q\} \Big\}$
\end{flushleft}
\begin{flushright}
$D(i,q,k,A,i')=\Big\{\beta=(s_i)\in\mathcal{C}(S) \mid A\subseteq
F_k^i\cap\{s_1,\ldots,s_q\}$, $F_k^{i'}\cap\{s_1,\ldots,s_q\}=A$ and
$F_{k'}^{i'}\cap\{s_1,\ldots,s_q\}\subseteq A$ for any $k'\neq k$
$\Big\}$.
\end{flushright}
The sets $E(i,q,k,A)$ and $D(i,q,k,A,i')$ are open, therefore
$\mathcal{A}$ is a Borel subset of $\mathcal{C}(S)$. By Stern's
theorem, there is a subtree $T$ of $S$ such that either
(a)~$\mathcal{C}(T)\subseteq \mathcal{A}$ or
(b)~$\mathcal{C}(T)\cap\mathcal{A}=\emptyset$. By Theorem
\ref{th.Arvan.comb.} we must have (a). Since $\{F_k^i\}_{i\in J}$ is
dense in $\{F_k^i\}_{i\in I}$ for any $k=1,2,\ldots,n$, the tree $T$
satisfies the desired property.
\end{proof}

\section{Tree-families of continuous functions}\label{sec.Continuous
Functions} In this section we consider tree-families $(f_s)_{s\in
S}$ of continuous functions. Then, under some conditions, we show
that there exists a subtree $T$ of $S$ such that for any chain
$\beta\in\mathcal{C}(T)$, the sequence $(f_s)_{s\in\beta}$ is
unconditional. First, we prove the following general result.

\begin{theorem}\label{th.unconditional}
Let $(x_s)_{s\in S}$ be a family in a Banach space $X$ such that
$x_s\neq0$, $s\in S$. Suppose that for any chain $\beta$ of $S$ the
sequence $(x_s)_{s\in\beta}$ contains an unconditional subsequence.
Then there exist a subtree $T$ of $S$ and a constant $C>0$ such that
for any chain $\beta$ of $\mathcal{T}$, the sequence
$(x_s)_{s\in\beta}$ is $C$-unconditional.
\end{theorem}

\begin{proof}
Consider the following subset of $\mathcal{C}(S)$:
$$\mathcal{A}=\{\beta\in\mathcal{C}(S) \mid (x_s)_{s\in\beta} \text{ is unconditional}
\}.$$

\begin{claim} The set $\mathcal{A}$ is a Borel subset of
$\mathcal{C}(S)$.
\end{claim}

Indeed, we observe that

\begin{itemize}
\item[{}] $\beta=(s_i)\in\mathcal{A}  \Leftrightarrow (x_s)_{s\in\beta}$ is
unconditional
\item[{}] $\Leftrightarrow$ there is $C>0$ such
that for any $n\in \mathbb{N}$, any
$(a_1,\ldots,a_n)\in\mathbb{R}^n$ and any
$A\subseteq\{1,\ldots,n\}$, $\left\|\sum_{i\in
A}a_ix_{s_i}\right\|\le C\left\|\sum^n_{i=1}a_ix_{s_i}\right\|$
\item[{}] $\Leftrightarrow$ there is $C\in \mathbb{Q}^+$ such that for any
$n\in\mathbb{N}$, any $q=(q_1,\ldots,q_n)\in\mathbb{Q}^n$ and any
$A\subseteq\{1,\ldots,n\}$, $\left\|\sum_{i\in
A}q_ix_{s_i}\right\|\le C\left\|\sum^n_{i=1}q_ix_{s_i}\right\|.$
\end{itemize}
Therefore
$$ \mathcal{A}=\bigcup_{C\in\mathbb{Q}^+}\bigcap_{n\in\mathbb{N}}\bigcap_{q\in\mathbb{Q}^n}\bigcap_{\mathcal{A}\subseteq\{1,\ldots,n\}}
D(C,n,q,A),$$ where, if $q=(q_1,\ldots,q_n)$, then
$$D(C,n,q,A)=\Big\{\beta=(s_i)\in\mathcal{C}(S)\mid
\Big\|\sum_{i\in A}q_ix_{s_i}\Big\|\le
C\Big\|\sum^n_{i=1}q_ix_{s_i}\Big\|\Big\}.
$$
Clearly, $D(C,n,q,A)$ is an open subset of $\mathcal{C}(S)$,
therefore $\mathcal{A}$ is Borel.

Stern's theorem implies that there exists a subtree $S'$ of $S$ such
that either (a)~$\mathcal{C}(S')\subseteq\mathcal{A}$ or
(b)~$\mathcal{C}(S')\cap\mathcal{A}=\emptyset$. By our hypotheses,
if $\beta$ is any chain of $S'$, then there is a subchain
$\alpha\subset\beta$ such that $(x_s)_{s\in \alpha}$ is an
unconditional sequence. Therefore, $\alpha\in\mathcal{A}$ and the
case (b) is impossible. Hence, we have that
$\mathcal{C}(S')\subseteq\mathcal{A}$, that is for any chain $\beta$
of $S'$ the sequence $(x_s)_{s\in\beta}$ is unconditional.

It remains to prove that we can find a subtree $T$ of $S'$ such that
the sequences $(x_s)_{s\in\beta}$, $\beta\in\mathcal{C}(T)$, share
the same unconditional constant $C$. To avoid introducing additional
notation, we assume that for any chain $\beta$ of the original tree
$S$, $(x_s)_{s\in\beta}$ is unconditional, that is
$\mathcal{C}(S)=\mathcal{A}$. As above, we have:
\begin{align*}
    \mathcal{C}(S)&=\bigcup_{C\in\mathbb{Q}^+}\bigcap_{n\in\mathbb{N}}\bigcap_{q\in\mathbb{Q}^n}\bigcap_{\mathcal{A}\subseteq\{1,\ldots,n\}}
D(C,n,q,A)\\
&=\bigcup_{C\in\mathbb{Q}^+} \mathcal{A}_C
\end{align*}
where $\mathcal{A}_C=\bigcap_{n,q,A}D(C,n,q,A)$ is the set of all
chains $\beta$ such that $(x_s)_{s\in\beta}$ is $C$-unconditional.
It is easy to see that $D(C,n,q,A)$ is also a closed subset of
$\mathcal{C}(S)$. Hence, $\mathcal{A}_C$ is a closed set and
$\mathcal{C}(S)=\bigcup_{C\in\mathbb{Q}^+} \mathcal{A}_C$ is
$F_\sigma$. Therefore, the Baire category theorem implies that there
exists a constant $C$ such that the set $\mathcal{A}_C$ has
non-empty interior. This means that there are finitely many nodes
$s_1<s_2<\ldots<s_m$ such that for any chain $\beta$ beginning with
$s_1,s_2,\ldots,s_m$, the sequence $(x_s)_{s\in\beta}$ is
$C$-unconditional. Let $T$ be the subtree consisting of the node
$s_m$ and all its followers. Clearly, $\mathcal{C}(T)\subseteq
\mathcal{A}_C$.

\end{proof}

Combining Theorem \ref{th.unconditional} with Theorems
\ref{th.Arvan.1} and \ref{th.Arvan.2}, we obtain the following
results.

\begin{theorem}\label{th.unconditional-1}
Let $K$ be a Hausdorff compact space, $X$ a Banach space and
$(f_s)_{s\in S}$, $f_s:K\to X$, a normalized family of continuous
functions. We assume that for any maximal chain
$\beta\in\mathcal{C}(S)$, the sequence $(f_s)_{s\in\beta}$ converges
pointwise to zero and that there exists a positive integer $J_\beta$
such that $card(f_s[K])\le J_\beta$ for any $s\in\beta$. Then there
exist a subtree $T$ of $S$ and a constant $C\ge1$ such that for any
chain $\beta$ of $T$, the sequence $(f_s)_{s\in\beta}$ is
$C$-unconditional.
\end{theorem}

\begin{theorem}\label{th.unconditional-2}
Let $K$ be a Hausdorff compact space and let $(f_s)_{s\in S}$,
$f_s:K\to\mathbb{R}^m$ be a family of continuous functions. We
assume that for any maximal chain $\beta$ of $S$, the sequence
$(f_s)_{s\in\beta}$ is uniformly bounded and converges pointwise to
zero. Furthermore, we assume that there are a null sequence
$(\epsilon^\beta_n)_{n\in\mathbb{N}}$ of positive real numbers and a
constant $\mu^\beta>0$ such that for any $s\in\beta$ and any $x\in
K$ either $\|f_s(x)\|\le\epsilon^\beta_{lev(s)}$ or
$\|f_s(x)\|\ge\mu^\beta$. Then there exist a subtree $T$ of $S$ and
a constant $C\ge1$ such that for any chain $\beta$ of $T$,
$(f_s)_{s\in\beta}$ is a $C$-unconditional sequence.
\end{theorem}

\section{The case of nearly unconditionality}\label{sec.Nearly}
In this section we prove the analogous to Elton's theorem for the
case of tree-families. Further, as in Theorem
\ref{th.unconditional}, we obtain a uniformity of the constants on
the chains. More precisely, we have the following.

\begin{theorem}\label{th.nearly-uncon}
Let $(x_s)_{s\in S}$ be a normalized family in a Banach space $X$.
Assume that for every chain $\beta\in\mathcal{C}(S)$, the sequence
$(x_s)_{s\in\beta}$ is weakly null. Then there exists a subtree $T$
of $S$ with the following property: for every $\delta>0$ there
exists $C=C(\delta)>0$ such that for any chain $\beta=(s_i)$ of $T$,
any $n\in\mathbb{N}$, any $a_1,\ldots,a_n\in[-1,1]$ and any
$F\subseteq\{i\leq n \mid |a_i|>\delta\}$,
$$
\Big\|\sum_{i\in F}a_ix_{s_i}\Big\|\leq C(\delta) \Big\|
\sum^n_{i=1}a_ix_{s_i}\Big\|. $$ That is, for any chain
$\beta\in\mathcal{C}(T)$, the sequence $(x_s)_{s\in\beta}$ is nearly
unconditional and the constant $C=C(\delta)$ is independent of the
chain $\beta$.
\end{theorem}

It is well-known that any normalized weakly null sequence in a
Banach space contains a Schauder basic subsequence. The proof of
this result can be easily transferred to tree-families. Thus we
obtain the next lemma whose proof is omitted.

\begin{lemma}
Suppose that $(x_s)_{s\in S}$ is a normalized family in a Banach
space, such that for any chain $\beta$ of $S$ the sequence
$(x_s)_{s\in\beta}$ is weakly null. Then, for every $\epsilon>0$
there exists a subtree $T$ of $S$ such that for any chain $\beta$ of
$T$, $(x_s)_{s\in\beta}$ is $(1+\epsilon)$-basic.
\end{lemma}

\begin{proof}[Proof of Theorem \ref{th.nearly-uncon}]
We may assume, by passing to a subtree if necessary, that for any
chain $\beta$ of $S$, $(x_s)_{s\in\beta}$ is a basic sequence with
basis constant $D\geq 1$, where $D$ is an absolute constant.

We consider the following subset of $\mathcal{C}(S)$:
$$\mathcal{A}=\{\beta\in\mathcal{C}(S) \mid (x_s)_{s\in\beta} \text{ is nearly unconditional}\}.
$$
Now we observe that:
\begin{itemize}
\item[{}] $\beta=(s_i)\in\mathcal{A}  \Leftrightarrow (x_s)_{s\in\beta}$
is nearly unconditional
\item[{}] $\Leftrightarrow$ for every $\delta>0$ there exists $C=C(\delta,\beta)>0$ such
that for any $n\in \mathbb{N}$, any $(a_1,\ldots,a_n)\in [-1,1]^n$
and any $F\subseteq\{i\le n \mid |a_i|>\delta\}$, $\left\|\sum_{i\in
F}a_ix_{s_i}\right\|\le C\left\|\sum^n_{i=1}a_ix_{s_i}\right\|$
\item[{}] $\Leftrightarrow$ for every $\delta\in\mathbb{Q}^+$ there exists $C=C(\delta,\beta)\in \mathbb{Q}^+$ such that for any
$n\in\mathbb{N}$, any
$q=(q_1,\ldots,q_n)\in(\mathbb{Q}\cap[-1,1])^n$ and any
$F\subseteq\{i\le n \mid |q_i|>\delta\}$, $\left\|\sum_{i\in
F}q_ix_{s_i}\right\|\le C\left\|\sum^n_{i=1}q_ix_{s_i}\right\|.$
\end{itemize}
Therefore
$$ \mathcal{A}=\bigcap_{\delta\in\mathbb{Q}^+}\bigcup_{C\in\mathbb{Q}^+}\bigcap_{n\in\mathbb{N}}
\bigcap_{q\in(\mathbb{Q}\cap[-1,1])^n}\bigcap_{F\subseteq\{i\le n
\mid |q_i|>\delta\}} D(C,n,q,F),$$ where, if $q=(q_1,\ldots,q_n)$,
then
$$D(C,n,q,F)=\Big\{\beta=(s_i)\in\mathcal{C}(S)\mid
\Big\|\sum_{i\in F}q_ix_{s_i}\Big\|\le
C\Big\|\sum^n_{i=1}q_ix_{s_i}\Big\|\Big\}.
$$
Clearly, $D(C,n,q,F)$ is an open subset of $\mathcal{C}(S)$ and
hence $\mathcal{A}$ is a Borel set. Stern's theorem implies that
there is a subtree $S'$ of $S$ such that either
(a)~$\mathcal{C}(S')\subseteq\mathcal{A}$ or
(b)~$\mathcal{C}(S')\cap\mathcal{A}=\emptyset$. By Theorem
\ref{th.Elton} every normalized, weakly null sequence contains a
nearly unconditional subsequence, therefore the case (b) is
impossible. Thus $\mathcal{C}(S')\subseteq\mathcal{A}$, that is for
any chain $\beta$ of $S'$ the sequence $(x_s)_{s\in\beta}$ is nearly
unconditional.

Now assume that for the original tree $S$ we have
$\mathcal{C}(S)\subseteq\mathcal{A}$. It remains to show that there
is a subtree $T$ of $S$ such that for every $\delta>0$ there is
$C=C(\delta)>0$ such that for any chain $\beta=(s_i)$ of $T$, any
$n\in\mathbb{N}$, any $a_1,\ldots,a_n\in[-1,1]$ and any
$F\subseteq\{i\le n \mid |a_i|> \delta\}$, $\|\sum_{i\in
F}a_ix_{s_i}\| \le C(\delta)  \|\sum^n_{i=1}a_ix_{s_i} \|$. That is,
the constant $C$ is independent of $\beta$.

We start with the following observation. Fix some positive number
$\delta$. Since for any $\beta\in\mathcal{C}(S)$,
$(x_s)_{s\in\beta}$ is nearly unconditional, it follows that
$$\mathcal{C}(S)=\bigcup_{C\in\mathbb{Q}^+}\bigcap_{n\in\mathbb{N}}
\bigcap_{q\in(\mathbb{Q}\cap[-1,1])^n}\bigcap_{F\subseteq\{i\le n
\mid |q_i|>\delta\}} D(C,n,q,F).$$ As in the proof of Theorem
\ref{th.unconditional}, the Baire category theorem implies that
there exist a positive constant $C(\delta)$ and a subtree $T$ of $S$
such that: for any $\beta=(s_i)\in\mathcal{C}(T)$, any
$n\in\mathbb{N}$, any $a_1,\ldots,a_n\in[-1,1]$ and any
$F\subseteq\{i\le n \mid |a_i|>\delta\}$, $\|\sum_{i\in
F}a_ix_{s_i}\|\le C(\delta) \|\sum^n_{i=1}a_ix_{s_i} \|$. Therefore
the subtree $T$ satisfies the desired property, however for the
specific number $\delta$.

In order to obtain the general result for arbitrary $\delta>0$, we
consider a null sequence $(\delta_n)$ and we apply a diagonal-type
argument for the dyadic tree. The desired subtree $T$ is constructed
inductively. We quote the first steps.

Let $\delta$ be equal to 1. By the previous observation there are a
subtree $T_\emptyset$ of $S$ and a positive constant $R(1)$ such
that: for any chain $\beta=(s_i)\in\mathcal{C}(T_\emptyset)$, any
$n\in\mathbb{N}$, any $a_1,\ldots,a_n\in[-1,1]$ and any
$F\subseteq\{i\le n \mid |a_i|>1\}$, $\|\sum_{i\in F}a_ix_{s_i}\|
\le R(1) \|\sum^n_{i=1}a_ix_{s_i} \|$. Let $t_\emptyset$ be the
minimum element of $T_\emptyset$ and $t_0,t_1$ the nodes placed on
the first level of $T_\emptyset$. Then $t_\emptyset$ is the minimum
node of $T$ and $t_0,t_1$ complete the first level of $T$.

Let $\delta$ be equal to $1/2$ and let $\widetilde{T_0}$ be the
subtree of $T_\emptyset$ which contains the node $t_0$ and all its
followers in $T_\emptyset$. By the previous observation, we find a
subtree $T_0\subseteq\widetilde{T_0}\subseteq T_\emptyset$ and a
constant $R_1(1/2)$ such that for any chain $\beta=(s_i)$ of $T_0$,
any $n\in\mathbb{N}$, any $a_1,\ldots,a_n\in[-1,1]$ and any
$F\subseteq\{i\le n \mid |a_i|> 1/2\}$, $\|\sum_{i\in F}a_ix_{s_i}\|
\le R_1(1/2) \|\sum^n_{i=1}a_ix_{s_i} \|$. Let $t_{(0,0)},t_{(0,1)}$
be the nodes placed on the first level of $T_0$. Then
$t_{(0,0)},t_{(0,1)}$ are the successors of $t_0$ in $T$.

The subtree $T_1\subseteq T_\emptyset$, the constant $R_2(1/2)$ and
the nodes $t_{(1,0)},t_{(1,1)}$ are defined in a similar way. We
also set $R(1/2)=\max\{R_1(1/2),R_2(1/2)\}$ and the second level of
$T$ has been completed.

We inductively construct a subtree $T$ of $S$ and positive constants
$R(1/k)$, $k\in\mathbb{N}$, satisfying the following property: for
any chain $\beta= (s_i )^\infty_{i=1}$ of $T$ with $lev_T(s_1)\ge
k$, any $n\in\mathbb{N}$, any $a_1,\ldots,a_n\in[-1,1]$ and any
$F\subseteq\{i\le n \mid |a_i|> 1/k\}$
$$
\Big\|\sum_{i\in F}a_ix_{s_i}\Big\|\le R(1/k)
\Big\|\sum^n_{i=1}a_ix_{s_i}\Big\|.
$$

\begin{claim} The subtree $T$ satisfies the conclusion
of the theorem.
\end{claim}

For any level $r=1,2,\ldots$, let $A_1,A_2,\ldots,A_{2^r}$ be an
enumeration of the maximal linearly ordered subsets of $T$ which
contain nodes of level less or equal to $r$. We set $B(r)=\max\{c_i
\mid 1\le i\le 2^r\}$, where $c_i$ is the unconditional constant of
the finite sequence $\{x_s \mid s\in A_i\}$. Therefore, for any
maximal chain $\beta=(s_i)_{i=1}^\infty$ of $T$, any
$I\subseteq\{1,2,\ldots,r+1\}$ and any scalars $a_1,\ldots,a_{r+1}$
we have
$$
\Big\|\sum_{i\in I}a_ix_{s_i}\Big\|\le B(r)
\Big\|\sum^{r+1}_{i=1}a_ix_{s_i}\Big\|.$$

Suppose now that $k$ is a positive integer. We show that there is a
constant $C(1/k)>0$ depending only on $k$ such that for any chain
$\beta= (s_i )\in\mathcal{C}(T)$, any $n\in\mathbb{N}$, any
$a_1,\ldots,a_n\in[-1,1]$ and any $F\subseteq\{i\le n \mid |a_i|>
1/k\}$, $ \|\sum_{i\in F}a_ix_{s_i}\| \le R(1/k)
 \|\sum^n_{i=1}a_ix_{s_i} \|$. It suffices to consider only the
maximal chains of $T$. So, if $\beta=(s_i)^\infty_{i=1}$ is maximal
then

\begin{align*}
\Big\|\sum_{i\in F}a_ix_{s_i}\Big\| & \le \Big\|\sum_{i\in F,i\le
k}a_ix_{s_i}\Big\|+ \Big\|\sum_{i\in F,i>
k}a_ix_{s_i}\Big\| \\
&\le
B(k-1)\Big\|\sum^{k}_{i=1}a_ix_{s_i}\Big\|+R(1/k)\Big\|\sum^n_{i=k}a_ix_{s_i}\Big\|
\\
&\le B(k-1)D \Big\|\sum^n_{i=1}a_ix_{s_i}\Big\| +R(1/k)2D
\Big\|\sum^n_{i=1}a_ix_{s_i}\Big\|
\\
&=(B(k-1) D+R(1/k) 2D) \Big\|\sum^n_{i=1}a_ix_{s_i}\Big\|.
\end{align*}
The choice $C(1/k)=B(k-1)  D+R(1/k)  2D$ completes the proof.
\end{proof}

The next corollary is a consequence of Theorem
\ref{th.nearly-uncon}. We refer to \cite{Odell} for the proof of the
analogous result concerning sequences indexed by natural numbers.

\begin{corollary}
Let $(x_s)_{s\in S}$ be a normalized family in a Banach space $X$,
such that for every chain $\beta\in\mathcal{C}(S)$ the sequence
$(x_s)_{s\in\beta}$ is weakly null. We further assume that for no
chain $\beta\in\mathcal{C}(S)$, the sequence $(x_s)_{s\in\beta}$ is
equivalent to the unit vector basis of $c_0$. Then there exists a
subtree $T$ of $S$ such that for any chain
$\beta=(s_i)\in\mathcal{C}(T)$, $(x_s)_{s\in\beta}$ is a
semi-boundedly complete basic sequence, that is whenever
$\sup_n\|\sum_{i=1}^n\lambda_ix_{s_i}\|<+\infty$ then
$\lim_{n\to\infty}\lambda_n=0$.

\end{corollary}

\section{The case of convex unconditionality}\label{sec.Convex}
We now use the techniques of the previous sections in the case of
convex unconditionality. As a result, we prove the analogous to
Theorem \ref{th.AMT} for tree-families.

\begin{theorem}\label{th.convex}
Let $(x_s)_{s\in S}$ be a normalized tree-family in a Banach space
$X$. Assume that for each chain $\beta\in\mathcal{C}(S)$, the
sequence $(x_s)_{s\in\beta}$ is weakly null. Then, there exists a
subtree $T$ of $S$ with the following property: for every $\delta>0$
there exists a constant $C=C(\delta)>0$ such that for any chain
$\beta=(s_i)$ of $T$, any absolutely convex combination
$x=\sum^\infty_{n=1}a_nx_{s_n}$ with $\|x\|\ge\delta$ and any
sequence $(\varepsilon_n)_{n\in\mathbb{N}}$ of signs,
$$
\Big\|\sum^\infty_{n=1}\varepsilon_na_nx_{s_n}\Big\|\ge C(\delta).
$$
That is, for any chain $\beta\in\mathcal{C}(T)$, $(x_s)_{s\in\beta}$
is a convexly unconditional sequence and the constant $C=C(\delta)$
depends only on $\delta$.
\end{theorem}

\begin{proof}
The proof follows the lines of the proof of Theorem
\ref{th.nearly-uncon}. We assume that for any chain $\beta$ of $S$
the sequence $(x_s)_{s\in\beta}$ is $D$-basic. We consider the
following subset of $\mathcal{C}(S)$:
$$\mathcal{A}=\{\beta\in\mathcal{C}(S) \mid (x_s)_{s\in\beta} \text{ is convexly unconditional}\}.
$$
We observe that:
\begin{itemize}
\item[{}] $\beta=(s_i)\in\mathcal{A}  \Leftrightarrow (x_s)_{s\in\beta}$
is convexly unconditional
\item[{}] $\Leftrightarrow$ for every $\delta>0$ there is $C=C(\delta,\beta)>0$ such
that for any $(a_n)\in \mathbb{R}^\mathbb{N}$ with
$\sum_{n=1}^\infty |a_n|=1$ and $\|\sum_{n=1}^\infty
a_nx_{s_n}\|\geq \delta$ and any sequence
$(\varepsilon_n)\in\{-1,1\}^\mathbb{N}$, $\left\|\sum_{n=1}^\infty
\varepsilon_na_nx_{s_n}\right\|\ge C$

\item[{}] $\Leftrightarrow$ for every $\delta>0$ there is $C=C(\delta,\beta)>0$ such
that for any $N\in\mathbb{N}$, any $(a_n)_{n=1}^N\in \mathbb{R}^N$
with $\sum_{n=1}^N |a_n|=1$ and $\|\sum_{n=1}^N a_nx_{s_n}\|\geq
\delta$ and any $(\varepsilon_n)_{n=1}^N\in\{-1,1\}^N$,
$\left\|\sum_{n=1}^N \varepsilon_na_nx_{s_n}\right\|\ge C$

\item[{}] $\Leftrightarrow$ for every $\delta\in\mathbb{Q}^+$ there is $C=C(\delta,\beta)\in\mathbb{Q}^+$ such
that for any $N\in\mathbb{N}$, any $q=(q_n)_{n=1}^N\in \mathbb{Q}^N$
with $\sum_{n=1}^N |q_n|=1$, $\|\sum_{n=1}^N q_nx_{s_n}\|\geq
\delta$ and any $\varepsilon=(\varepsilon_n)_{n=1}^N\in\{-1,1\}^N$,
$\left\|\sum_{n=1}^N \varepsilon_nq_nx_{s_n}\right\|\ge C$.
\end{itemize}
Therefore
$$ \mathcal{A}=\bigcap_{\delta\in\mathbb{Q}^+}\bigcup_{C\in\mathbb{Q}^+}\bigcap_{N\in\mathbb{N}}
\bigcap_{\underset{\sum|q_n|=1}{q\in\mathbb{Q}^N}}\Big[E(\delta,N,q)\bigcup\Big(\bigcap_{\varepsilon\in\{-1,1\}^N}
D(\delta,C,N,q,\varepsilon)\Big)\Big],$$ where,
\begin{align*}
    E(\delta,N,q) &= \Big\{\beta=(s_i)\in\mathcal{C}(S)\mid
    \Big\|\sum_{n=1}^Nq_nx_{s_n}\Big\|<\delta\Big\}\\
    D(\delta,C,N,q,\varepsilon) &=
    \Big\{\beta=(s_i)\in\mathcal{C}(S)\mid
    \Big\|\sum_{n=1}^Nq_nx_{s_n}\Big\|\ge\delta \text{ and } \Big\|\sum_{n=1}^N \varepsilon_nq_nx_{s_n}\Big\|\ge
    C \Big\}.
\end{align*}

It follows that $\mathcal{A}$ is a Borel subset of $\mathcal{C}(S)$.
Stern's theorem implies that there is a subtree $S'$ of $S$ such
that either (a)~$\mathcal{C}(S')\subseteq\mathcal{A}$ or
(b)~$\mathcal{C}(S')\cap\mathcal{A}=\emptyset$. By Theorem
\ref{th.AMT} we must have (a), that is for any chain $\beta$ of $S'$
the sequence $(x_s)_{s\in\beta}$ is convexly unconditional.

We next assume that for the original tree $S$ we have
$\mathcal{C}(S)\subseteq\mathcal{A}$ and we show that there is a
subtree $T$ of $S$ such that: for any $\delta>0$ there exists
$C=C(\delta)>0$ depending only on $\delta$ such that for any chain
$\beta=(s_i)$ of $T$, any absolutely convex combination
$x=\sum^\infty_{n=1}a_nx_{s_n}$ with $\|x\|\ge\delta$ and any
sequence $(\varepsilon_n)_{n\in\mathbb{N}}$ of signs,
$\|\sum^\infty_{n=1}\varepsilon_na_nx_{s_n}\|\ge C(\delta)$.

For a fixed $\delta>0$, as in the proof of Theorem
\ref{th.unconditional}, we find a subtree which satisfies the
desired property for the specific $\delta$. We next consider the
sequence $\delta_n=\frac{1}{2n}$ and using repeatedly the previous
observation we inductively construct a subtree $T$ of $S$ and
positive constants $R(1/k)$, $k\in\mathbb{N}$, such that: for any
$k\in\mathbb{N}$, any chain $\beta=(s_i)_{i=1}^\infty$ of $T$ with
$lev_T(s_1)\ge k$, any absolutely convex combination
$x=\sum^\infty_{n=1}a_nx_{s_n}$ with $\|x\|\ge1/2k$ and any
$(\varepsilon_n)_{n\in\mathbb{N}}\in\{-1,1\}^\mathbb{N}$,
$\|\sum^\infty_{n=1}\varepsilon_na_nx_{s_n}\|\geq R(1/k)$.

\begin{claim}
The subtree $T$ satisfies the conclusion of the theorem.
\end{claim}

Let $k\in\mathbb{N}$. We show that there is a constant $C(1/k)$
depending on $k$ such that: for any maximal chain
$\beta=(s_i)^\infty_{i=1}\in\mathcal{C}(T)$, any absolutely convex
combination $x=\sum^\infty_{n=1}a_nx_{s_n}$ with $\|x\|\geq 1/k$ and
any sequence $(\varepsilon_n)_{n\in\mathbb{N}}$ of signs,
$\|\sum_{n=1}^\infty\varepsilon_na_nx_{s_n}\|\geq C(1/k)$. We
distinguish the following two cases.

\emph{Case 1.} Suppose that $\left\|\sum_{n=1}^k
a_nx_{s_n}\right\|\ge\frac{1}{2k}$. Then we have

\begin{align*}
\Big\|\sum^\infty_{n=1}\varepsilon_na_nx_{s_n}\Big\| &\ge
\frac{1}{D} \Big\|\sum_{n=1}^k
\varepsilon_na_nx_{s_n}\Big\| \\
&\ge \frac{1}{2B(k-1)
D} \Big\|\sum_{n=1}^k a_nx_{s_n}\Big\| \\
&\ge\frac{1}{2B(k-1) D} \,\, \frac{1}{2k}
\end{align*}
where the constant $B(k-1)$ has been defined in the proof of Theorem
\ref{th.nearly-uncon}.

\emph{Case 2.} Suppose that
$\left\|\sum^\infty_{n=k+1}a_nx_{s_n}\right\|\ge1/2k$. We set
$a=\sum^\infty_{n=k+1}|a_n|$ and
$x=\sum^\infty_{n=k+1}\frac{a_n}{a}\, x_{s_n}$. Then $1/2k\le a\le1$
and $x$ is an absolutely convex combination of
$(x_{s_n})^\infty_{n=k+1}$ such that $\|x\|=\frac{1}{a}
\left\|\sum^\infty_{n=k+1}a_nx_{s_n}\right\|\ge\frac{1}{2k}$. By the
construction of $T$, it follows that $\left\|\sum^\infty_{n=k+1}
\varepsilon_n\frac{a_n}{a}x_{s_n}\right\|\geq R(1/k)$. Therefore

\begin{align*}
\Big\|\sum^\infty_{n=1}\varepsilon_na_nx_{s_n}\Big\|
 & \ge\frac{1}{2D} \Big\|\sum^\infty_{n=k+1}\varepsilon_na_nx_{s_n}\Big\|
\\
&=\frac{a}{2D} \Big\|\sum^\infty_{n=k+1}\varepsilon_n\frac{a_n}{a}
x_{s_n}\Big\| \\
&\ge\frac{1}{2D}\, \frac{1}{2k}\, R(1/k).
\end{align*}
The choice $C(1/k)=\min\left\{\frac{1}{2B(k-1) D} \,
\frac{1}{2k}\,,\,\frac{1}{2D}\, \frac{1}{2k}\,  R(1/k)\right\}$
completes the proof.
\end{proof}

\section{A dichotomy result for more general
tree-families}\label{sec.general}
 In this section, we consider the
more general setting, where $(x_s)_{s\in S}$ is a normalized
tree-family such that for any chain $\beta$ of $S$,
$(x_s)_{s\in\beta}$ is a Schauder basic sequence. For such families
we prove the following dichotomy theorem. Recall that a normalized
Schauder basis $(e_n)$ is called \emph{semi-boundedly} complete if
for every sequence $(\lambda_i)\in\mathbb{R}^\mathbb{N}$, the
condition $\sup_n\|\sum_{i=1}^n\lambda_ie_i\|<+\infty$ implies that
$\lim_{n\to+\infty}\lambda_n=0$.

\begin{theorem}\label{th.general}
Let $(x_s)_{s\in S}$ be a normalized tree-family in a Banach space
$X$ such that for any chain $\beta\in\mathcal{C}(S)$, the sequence
$(x_s)_{s\in\beta}$ is Schauder basic. Then there exists a subtree
$T$ of $S$ such that: either
\begin{enumerate}
    \item for any chain $\beta\in\mathcal{C}(T)$, the sequence
    $(x_s)_{s\in\beta}$ is semi-boundedly complete; or
    \item for no chain $\beta\in\mathcal{C}(T)$, the sequence
    $(x_s)_{s\in\beta}$ is semi-boundedly complete.
\end{enumerate}
\end{theorem}

\begin{proof}
We consider the following subset of $\mathcal{C}(S)$
$$\mathcal{A}=\{\beta\in\mathcal{C}(S) \mid (x_s)_{s\in \beta} ~
\text{is semi-boundedly complete}\}.$$

\begin{claim}
The set $\mathcal{A}$ is co-analytic.
\end{claim}
In particular, we prove that the complement of $\mathcal{A}$ is an
analytic subset of $\mathcal{C}(S)$. To this end, we consider the
space $\mathcal{C}(S)\times \mathbb{R}^\mathbb{N}$ endowed with the
product topology and we set
\begin{flushright}
$\mathcal{F}=\{(\beta,\lambda)\in \mathcal{C}(S)\times
\mathbb{R}^\mathbb{N} \mid \beta=(s_i)$,  $\lambda=(\lambda_i)$ is a
bounded sequence not converging to $0$ such that
$\sup_n\|\sum_{i=1}^n\lambda_i x_{s_i}\|<+\infty \}$.
\end{flushright}
Clearly,
$\mathcal{C}(S)\setminus\mathcal{A}=\text{proj}_1(\mathcal{F})$,
where $\text{proj}_1$ denotes the projection
$\text{proj}_1:\mathcal{C}(S)\times \mathbb{R}^\mathbb{N} \to
\mathcal{C}(S)$. Therefore, it suffices to show that $\mathcal{F}$
is a Borel subset of $\mathcal{C}(S)\times \mathbb{R}^\mathbb{N}$.

Now, we write
$$\mathcal{F}=\mathcal{F}_1\cap\mathcal{F}_2,$$
where
\begin{align*}
    \mathcal{F}_1 &= \{ (\beta,\lambda)\in \mathcal{C}(S)\times
\mathbb{R}^\mathbb{N} \mid \beta=(s_i),~  \lambda=(\lambda_i)
~\text{and} ~ \sup_n\|\sum_{i=1}^n\lambda_i x_{s_i}\|<+\infty \}\\
\mathcal{F}_2 &= \{ (\beta,\lambda)\in \mathcal{C}(S)\times
\mathbb{R}^\mathbb{N} \mid  \lambda=(\lambda_i) ~\text{is a bounded
sequence not converging to} ~ 0\}\\
&= \mathcal{C}(S) \times \{\lambda \in \mathbb{R}^\mathbb{N} \mid
\lambda=(\lambda_i) ~\text{is a bounded sequence not converging to}
~ 0\}.
\end{align*}

First, we argue that $\mathcal{F}_1$ is a Borel set. Indeed, we have
\begin{align*}
    (\beta,\lambda)\in \mathcal{F}_1 & \Leftrightarrow \sup_n\|\sum_{i=1}^n\lambda_i
    x_{s_i}\|<+\infty\\
    &\Leftrightarrow (\exists M>0) (\forall n\in\mathbb{N}) \Big[\|\sum_{i=1}^n\lambda_i
    x_{s_i}\|\le M\Big]\\
    & \Leftrightarrow (\exists M\in\mathbb{Q}^+) (\forall n\in\mathbb{N}) \Big[\|\sum_{i=1}^n\lambda_i
    x_{s_i}\|\le M\Big].
\end{align*}
Therefore,
$$\mathcal{F}_1=\bigcup_{M\in\mathbb{Q}^+} \bigcap_{n\in\mathbb{N}}
\mathcal{G}_{M,n}$$
 where $\mathcal{G}_{M,n}=\{(\beta,\lambda)\in \mathcal{C}(S)\times
\mathbb{R}^\mathbb{N} \mid \|\sum_{i=1}^n\lambda_i x_{s_i}\|\le
M\}$. Clearly, $\mathcal{G}_{M,n}$ is an open set and hence
$\mathcal{F}_1$ is a Borel set.

Next, we observe that $\mathcal{F}_2$ is also a Borel set. Indeed,
it is enough to write down the following:
\begin{align*}
    \lambda=(\lambda_i)~ \text{is a bounded sequence} &\Leftrightarrow
(\exists M\in\mathbb{Q}^+) (\forall i\in\mathbb{N})
\Big[\abs{\lambda_i}\le M\Big]\\
\lambda=(\lambda_i)~ \text{does not converge to} ~0 &\Leftrightarrow
(\exists \epsilon\in\mathbb{Q}^+)(\forall i_0\in\mathbb{N})(\exists
i\ge i_0) [\abs{\lambda_i}>\epsilon].
\end{align*}
It is easy now to see that $\mathcal{F}_2$ is a Borel set and the
proof of the claim is complete.

Since $\mathcal{A}$ is co-analytic, we can apply Stern's theorem. It
follows that there exists a subtree $T$ of $S$ such that either
$\mathcal{C}(T)\subset \mathcal{A}$ or
$\mathcal{C}(T)\cap\mathcal{A}=\emptyset$, that is either
\begin{enumerate}
    \item for any chain $\beta \in \mathcal{C}(T)$, the sequence
    $(x_s)_{s\in \beta}$ is semi-boundedly complete; or
    \item for no chain $\beta\in \mathcal{C}(T)$, the sequence
    $(x_s)_{s\in\beta}$ is semi-boundedly complete.
\end{enumerate}

\end{proof}

Suppose now that $(x_s)_{s\in S}$ is a normalized tree-family such
that for any chain $\beta\in\mathcal{C}(S)$, $(x_s)_{s\in\beta}$ is
weakly null. In this case, using Theorem \ref{th.nearly-uncon}, we
can improve the result of Theorem \ref{th.general} and we obtain the
following dichotomy.

\begin{theorem}\label{th.dichotomy-c0}
Let $(x_s)_{s\in S}$ be a normalized tree-family in a Banach space
$X$. We assume that for any chain $\beta\in\mathcal{C}(S)$, the
sequence $(x_s)_{s\in\beta}$ is weakly null. Then there exists a
subtree $T$ of $S$ such that either:
\begin{enumerate}
    \item for any chain $\beta\in\mathcal{C}(T)$, the sequence
    $(x_s)_{s\in\beta}$ is semi-boundedly complete; or
    \item for any chain $\beta\in\mathcal{C}(S)$, the sequence
    $(x_s)_{s\in\beta}$ is $C$-equivalent to the unit vector basis
    of $c_0$, where $C>0$ is a common constant.
\end{enumerate}

\end{theorem}

\begin{proof}
We may assume that for any chain $\beta\in\mathcal{C}(S)$, the
sequence $(x_s)_{s\in\beta}$ is $D$-basic. In view of Theorem
\ref{th.general}, it suffices to consider only the case where no
sequence $(x_s)_{s\in\beta}$, $\beta\in\mathcal{C}(S)$, is
semi-boundedly complete and then we have to prove that there exists
a subtree $T$ of $S$ such that for any chain
$\beta\in\mathcal{C}(T)$, $(x_s)_{s\in\beta}$ is equivalent to the
unit vector basis of $c_0$. Furthermore, by Theorem
\ref{th.nearly-uncon} we may assume that for any chain
$\beta\in\mathcal{C}(S)$, the sequence $(x_s)_{s\in\beta}$ is nearly
unconditional.

Our first step is to show that any chain $\beta$ of $S$ contains a
subchain $\alpha\subset \beta$, such that $(x_s)_{s\in\alpha}$ is
equivalent to the unit vector basis of $c_0$. The proof of this fact
is essentially contained in \cite{Odell}. However we shall give a
brief description.

Since $(x_s)_{s\in\beta}$ is not semi-boundedly complete, it follows
that there exists a bounded sequence
$(\lambda_i)\in\mathbb{R}^\mathbb{N}$ not converging to $0$ so that
$\sup_n\|\sum_{i=1}^n\lambda_ix_{s_i}\|<\infty$. Clearly, we may
assume that $\abs{\lambda_i}\le 1$ for every $i\in\mathbb{N}$. Let
$M=(m_i)\subset \mathbb{N}$ and $\delta>0$ be such that
$\abs{\lambda_{m_i}}\ge \delta$ for all $i\in\mathbb{N}$. Then we
set $\alpha=(s_{m_i})_{i\in\mathbb{N}}\subset \beta$ and we claim
that $(x_s)_{s\in\alpha}$ is equivalent to the unit vector basis of
$c_0$.

Firstly, by Theorem \ref{th.nearly-uncon}, we have
$$L:=\sup_n\|\sum_{i=1}^n\lambda_{m_i}x_{s_{m_i}}\|<+\infty$$
and further, for any signs $(\varepsilon_i)\in\{-1,1\}^n$,
$$\|\sum_{i=1}^n \varepsilon_i \lambda_{m_i}x_{s_{m_i}}\| \le 2
C(\delta) \|\sum_{i=1}^n\lambda_{m_i}x_{s_{m_i}}\| \leq 2 C(\delta)
L$$
 (where the constant $C(\delta)$ is given by Theorem
 \ref{th.nearly-uncon}). Therefore, for any
$n\in\mathbb{N}$, any $t_1,\ldots,t_n\in\mathbb{R}$ and any $f\in
X^\ast$, $\|f\|\le 1$, we obtain
\begin{align*}
    \abs{f(\sum_{i=1}^nt_ix_{s_{m_i}})} & \leq \sum_{i=1}^n
    \abs{t_i} \abs{f(x_{s_{m_i}})} \le (\max\abs{t_i}) \sum_{i=1}^n \varepsilon_i
    f(x_{s_{m_i}})\\
    &\leq \frac{1}{\delta} (\max\abs{t_i}) \sum_{i=1}^n
    \varepsilon_i \lambda_{m_i} f(x_{s_{m_i}}) \\
    & \leq \frac{1}{\delta} (\max\abs{t_i}) \|\sum_{i=1}^n
    \varepsilon_i \lambda_{m_i} x_{s_{m_i}}\|  \\
    &\leq \frac{1}{\delta} 2 C(\delta) L \max\abs{t_i}.
\end{align*}
It follows that
$$\Big\|\sum_{i=1}^nt_ix_{s_{m_i}}\Big\| \le
\frac{2C(\delta)L}{\delta} \max_i\abs{t_i}.$$

So far we have shown that any chain $\beta\in\mathcal{C}(S)$
contains a subchain $\alpha\subset \beta$ such that
$(x_s)_{s\in\alpha}$ is equivalent to the unit vector basis of
$c_0$. Now, we proceed as follows. We consider the set
$$\mathcal{A}=\{\beta\in \mathcal{C}(S) \mid (x_s)_{s\in\beta}
~\text{is equivalent to the unit vector basis of } c_0\}.$$
 and we observe that
 \begin{itemize}
\item[{}] $\beta=(s_i)\in\mathcal{A}  \Leftrightarrow (x_s)_{s\in\beta}$ is
equivalent to the unit vector basis of $c_0$
\item[{}] $\Leftrightarrow$ there is $C>0$ such
that for any $n\in \mathbb{N}$ and any
$(a_1,\ldots,a_n)\in\mathbb{R}^n$, $\left\|\sum_{i =1}^n
a_ix_{s_i}\right\|\le C\max \abs{a_i}$
\item[{}] $\Leftrightarrow$ there is $C\in \mathbb{Q}^+$ such that for any
$n\in\mathbb{N}$ and any $q=(q_1,\ldots,q_n)\in\mathbb{Q}^n$,
$\left\|\sum_{i=1}^n q_ix_{s_i}\right\|\le C\max \abs{q_i}.$
\end{itemize}
Hence
$$ \mathcal{A}=\bigcup_{C\in\mathbb{Q}^+}\bigcap_{n\in\mathbb{N}}\bigcap_{q\in\mathbb{Q}^n} D(C,n,q),$$
 where,
$$D(C,n,q)=\Big\{\beta=(s_i)\in\mathcal{C}(S)\mid
\Big\|\sum_{i=1}^n q_ix_{s_i}\Big\|\le C \max \abs{q_i}\Big\}.
$$
It follows that $\mathcal{A}$ is a Borel subset of $\mathcal{C}(S)$.
Therefore, by Theorem \ref{th.Stern} there exists a subtree $S'$ of
$S$ such that either (a)~$\mathcal{C}(S') \subset \mathcal{A}$ or
(b)~$\mathcal{C}(S') \cap \mathcal{A} =\emptyset$. However, the case
(b) must be excluded, since for any chain $\beta\in\mathcal{C}(S)$
the sequence $(x_s)_{s\in\beta}$ contains a subsequence equivalent
to the basis of $c_0$. Finally, as in the proof of Theorem
\ref{th.unconditional}, an application of the Baire category theorem
shows that we can find a further subtree $T$ of $S'$ and a constant
$C>0$ such that for any chain $\beta\in\mathcal{C}(T)$,
$(x_s)_{s\in\beta}$ is $C$-equivalent to the unit vector basis of
$c_0$.

\end{proof}

\bibliographystyle{amsplain}

\end{document}